\documentclass[11pt]{article}
\usepackage [dvips]{graphics}
\usepackage[centertags]{amsmath}
\usepackage{amsfonts}
\usepackage{amssymb}
\usepackage{makeidx}
\usepackage{bm}
\usepackage{newlfont}
\usepackage{amsthm} %add
\usepackage{stmaryrd}
\usepackage[]{titletoc}  % from here
\usepackage{mathrsfs}
\usepackage{stmaryrd}
 \usepackage{graphicx}  %%add
 \usepackage{xypic}
\usepackage{hyperref}
\usepackage[paperwidth=19.5cm,paperheight=27cm,top=2.9cm,bottom=2.9cm,right=2.7cm]{geometry}

\newlength{\defbaselineskip}
\newcommand{\setlinespacing}[1]%
           {\setlength{\baselineskip}{#1 \defbaselineskip}}

\theoremstyle{plain}
\newtheorem{thm}{Theorem}[section]
\newtheorem{cor}[thm]{Corollary}
\newtheorem{lem}[thm]{Lemma}
\newtheorem{prop}[thm]{Proposition}

\makeatletter\@addtoreset{equation}{section} \makeatother
\begin{document}

\title {Invariant subspaces of weighted Bergman spaces in infinitely many variables}
\author{
Hui Dan \quad Kunyu Guo \quad Jiaqi Ni
 }
\date{}
\maketitle \noindent\textbf{Abstract:}
This paper is concerned with polynomially generated multiplier invariant subspaces of the weighted Bergman space $A_{\bm{\beta}}^2$ in infinitely many variables. We completely classify these invariant subspaces under the unitary equivalence. Our results not only cover cases of both the Hardy space $H^{2}(\mathbb{D}_{2}^{\infty})$ and the Bergman space $A^{2}(\mathbb{D}_{2}^{\infty})$ in infinitely many variables, but also apply in finite-variable setting.


\section{Introduction}

Let $T$ be a bounded linear operator on a complex separable Hilbert space $\mathcal{H}$. An invariant subspace of $T$ is a closed subspace $M$ of $\mathcal{H}$ such that $TM\subset M$. The famous invariant subspace problem (ISP) asks whether for every bounded linear operator on a complex separable Hilbert space has a nontrivial invariant subspace. Up to now, it is still an open problem. Let $A^2(\mathbb{D})$ denote the Bergman space over the open unit disk $\mathbb{D}$. H. Bercovici, C. Foias and C. Pearcy \cite{BFP} showed that the ISP is equivalent to the following problem: if $M$, $N$ are invariant subspaces of the Bergman shift (i.e. the coordinate multiplication operator on $A^2(\mathbb{D})$) satisfying $N\subset M$ and $\dim M/N=\infty$, then is there another invariant subspace $L$ such that $N\varsubsetneqq L\varsubsetneqq M$?  Therefore, the research of the ISP can be reduced to the study of invariant subspaces of the Bergman shift.

Recall that two invariant subspaces $M$, $N$ of $T$ are unitarily equivalent, if there is a unitary operator $U: M\rightarrow N$ such that $UT|_M=T|_N U$. Two unitarily equivalent invariant subspaces are essentially the same. As well known, two invariant subspaces of the Bergman shift are unitarily equivalent only when they coincide \cite{Ri}. This is an intrinsic reason why the research of the ISP can be reduced to the study of invariant subspaces of the Bergman shift, and why people are interested in invariant subspaces of the Bergman shift. For the Hardy space $H^{2}(\mathbb{D})$ over the open unit disk, Beurling's remarkable theorem states that each invariant subspace $M$ of the Hardy shift (i.e. the coordinate multiplication operator on $H^{2}(\mathbb{D})$) has the form $M=\eta H^{2}(\mathbb{D})$ for some inner function $\eta\in H^{2}(\mathbb{D})$. Therefore, all nonzero invariant subspaces of the Hardy shift are unitarily equivalent. Under unitary equivalence, invariant subspaces of the Bergman shift and the Hardy shift present two extreme phenomena. This inspires us to investigate the unitary equivalence of joint invariant subspaces of some ``intertwined" spaces, that is, the tensor product of Bergman spaces and Hardy spaces.

In this paper, we will use the language of Hilbert modules to pursue our study. The theory of Hilbert modules, developed by R. Douglas and V. Paulsen \cite{DP}, provides an appropriate framework to the operator theory in function spaces. Let $\mathcal{H}$ be a complex Hilbert space, and  $\mathcal{A}$ a commutative  algebra over the complex  number field $\mathbb{C}$. By $\mathcal{A}$-Hilbert module $\mathcal{H}$, it means that we assign an algebraic homomorphism $\sigma: \mathcal{A}\to B(\mathcal{H})$. In this case,  the module action is given by
$$ah=\sigma(a)h,\quad a\in\mathcal{A}, \,\,h\in \mathcal{H},$$
where $B(\mathcal{H})$ denotes the algebra of all bounded linear operators on $\mathcal{H}$.
 A closed subspace of $\mathcal{H}$ invariant under this module action is called a ($\mathcal{A}$-Hilbert) submodule of $\mathcal{H}$. Basic examples of Hilbert modules are the Hardy space $H^2(\mathbb{D}^n)$ over the polydisk $\mathbb{D}^n$ and the Bergman space $A^{2}(\Omega)$ over a bounded domain $\Omega$ in $\mathbb{C}^n$. Both of them are $\mathcal{P}_n$-Hilbert modules, where $\mathcal{P}_n=\mathbb{C}[\zeta_1,\cdots,\zeta_n]$, the polynomial ring of $n$-complex variables. Submodules of $H^2(\mathbb{D}^n)$ (or $A^{2}(\Omega)$) are exactly the joint invariant subspaces of coordinate multiplication operators.

Recall that two submodules $M$, $N$ of an $\mathcal{A}$-Hilbert module $\mathcal{H}$ are unitarily equivalent, if there exists a unitary map $U$: $M\rightarrow N$, called unitary module map, such that $U(ah)=aUh$ whenever $a\in\mathcal{A}$, $h\in M$. As we mentioned, all nonzero submodules of the Hardy module $H^{2}(\mathbb{D})$ are unitarily equivalent. However, this fails in the Hardy module $H^2(\mathbb{D}^n)$ for $n>1$. For such cases of several variables, O. Agrawal, D. Clark, and R. Douglas \cite{ACD} showed two submodules of finite codimension in $H^{2}(\mathbb{D}^{n})$ are unitarily equivalent only when they coincide. Note that by Ahern-Clack's theorem \cite{AC}, all finite codimensional submodules of $H^{2}(\mathbb{D}^{n})$ are generated by polynomials. K. Yan \cite{Yan} studied the unitary equivalence of submodules of $H^{2}(\mathbb{D}^{n})$ generated by a single homogeneous polynomial. For the case of general polynomially generated submodules, K. Guo \cite{Guo1} gave a complete characterization of unitary equivalence by virtue of the so-called Beurling forms of polynomial ideals. This result is stated as follows: Two submodules $M$, $N$ of $H^2(\mathbb{D}^n)$, both generated by polynomials, are unitarily equivalent if and only if there exists an ideal $\mathcal{L}$ of $\mathcal{P}_{n}$ and a pair $p,q$ of polynomials with the same modulus on the distinguished boundary $\mathbb{T}^n$, such that $M$ and $N$ are the closures of $p\mathcal{L}$ and $q\mathcal{L}$, respectively. Unitary equivalence of general submodules of $H^2(\mathbb{D}^n)$ $(n>1)$ is extremely difficult to characterize and this problem remains unsolved. For more works on the classification of Hardy submodules, we refer readers to \cite{CD, CG, DPSY, DPY, DY, Guo1, GY, Izu, Yan, Yang1, Yang2}.

In the case of the Bergman module $A^{2}(\Omega)$, the corresponding problem seems easier. When $n=1$, S. Richter \cite{Ri} proved that two submodules of $A^{2}(\Omega)$ are unitarily equivalent only if they are equal. For multidimensional cases, M. Putinar \cite{Pu} proved an analogue of this result when $\Omega$ is a bounded pseudoconvex domain in $\mathbb{C}^n$. In X. Chen and Guo's book \cite{CG}, Putinar's result is generalized to general bounded domains in $\mathbb{C}^n$. These works imply the high rigidity of Bergman submodules.

In the present paper, we will consider submodules of the following weighted Bergman spaces in infinitely many variables.

Let $dA(z)$ denote the normalized area measure on $\mathbb{D}$. The weighted Bergman space, $A_{\beta}^{2}=A_{\beta}^{2}(\mathbb{D})\ (\beta>-1)$, consists of all holomorphic functions in $\mathbb{D}$ with
$$\|f\|_{A_{\beta}^{2}}=\left(\int_{\mathbb{D}}|f(z)|^{2}dA_{\beta}(z)\right)^{\frac{1}{2}}<\infty,$$
where $dA_{\beta}(z)=(\beta+1)(1-|z|^2)^\beta dA(z)$. When $\beta\rightarrow-1$, noticing that the measures $dA_{\beta}(z)$ on the closed unit disk converge to $\frac{d\theta}{2\pi}$ in the weak-star topology, the limit space of the weighted Bergman spaces $A_{\beta}^{2}$ is the Hardy space $H^{2}(\mathbb{D})$. We thus write $A_{-1}^{2}$ for the Hardy space $H^{2}(\mathbb{D})$.

Let $\mathbb{Z}_{+}^{\infty}$ be the set of all finitely supported sequences of nonnegative integers. For each $\alpha=(\alpha_{1},\cdots,\alpha_{m},0,0,\cdots)\in\mathbb{Z}_{+}^{\infty}$ and a sequence of complex numbers $\zeta=(\zeta_1,\zeta_2,\cdots)$, write $\zeta^{\alpha}=\zeta_{1}^{\alpha_{1}}\cdots\zeta_{m}^{\alpha_{m}}$. Let $\mathbb{N}=\{1,2,\cdots\}$ be the set of positive integers, and $\bm{\beta}=\{\beta_{n}\}_{n\in\mathbb{N}}$ a sequence in $[-1,\infty)$. We define the weighted Bergman space $A_{\bm{\beta}}^2$ in infinitely many variables to be exactly the tensor product of infinitely many usual weighted Bergman spaces (including $A_{-1}^{2}$). More precisely, $A_{\bm{\beta}}^2=\bigotimes_{n\in\mathbb{N}}A_{\beta_{n}}^{2}$ in infinitely many variables is defined as follows:
$$A_{\bm{\beta}}^2=\bigotimes\limits_{n\in\mathbb{N}}A_{\beta_{n}}^{2}=\left\{F=\sum\limits_{\alpha\in\mathbb{Z}_{+}^{\infty}}c_{\alpha}\zeta^{\alpha}: \|F\|^{2}=\sum\limits_{\alpha\in\mathbb{Z}_{+}^{\infty}}|c_{\alpha}|^{2}\omega_{\alpha}<\infty\right\},$$
where
$$\omega_{\alpha}=\prod_{n=1}^{\infty}\|\zeta_{n}^{\alpha_{n}}\|_{A_{\beta_{n}}^{2}}^{2}
=\prod_{n=1}^{\infty}\frac{\alpha_{n}!\Gamma(\beta_{n}+2)}{\Gamma(\alpha_{n}+\beta_{n}+2)},\quad \alpha=(\alpha_1,\alpha_2,\cdots)\in\mathbb{Z}_{+}^{\infty},$$
and $\Gamma$ stands for the usual Gamma function. When $\bm{\beta}=\{-1\}_{n\in\mathbb{N}}$,  $A_{\bm{\beta}}^2 $ is the Hardy space $H^{2}(\mathbb{D}_{2}^{\infty})$ over the Hilbert multidisk $\mathbb{D}_{2}^{\infty}$ (see \cite{Ni}); when $\bm{\beta}=\{0\}_{n\in\mathbb{N}}$, $A_{\bm{\beta}}^2 $ is the (unweighted) Bergman space $A^{2}(\mathbb{D}_{2}^{\infty})$ which is the direct limit of $\{A^{2}(\mathbb{D}^{n})\}_{n\in\mathbb{N}}$.

Set $\mathcal{P}_\infty=\bigcup_{n=1}^{\infty}\mathcal{P}_n$, the ring of polynomials in countably infinitely many complex variables for which each polynomial in $\mathcal{P}_\infty$ only depends on finitely many complex variables. It is easily seen that the weighted Bergman space $A_{\bm{\beta}}^2$ has a canonical structure of $\mathcal{P}_\infty$-Hilbert module, where the module action is defined by multiplications of polynomials in $\mathcal{P}_\infty$.

In this paper, our main result is Theorem \ref{thm301}, a complete classification under unitary equivalence for submodules of $A_{\bm{\beta}}^2$ generated by polynomials. As a consequence, we obtain a rigidity theorem for finite codimensional submodules of $A_{\bm{\beta}}^2$ in Corollary \ref{thm303}, which generalizes Agrawal, Clark and Douglas's result \cite{ACD}. We also give the characterization of the unitary equivalence of submodules of $H^{2}(\mathbb{D}_{2}^{\infty})$ generated by polynomials (see Corollary \ref{thm304}). Another corollary (see Corollary \ref{thm305}) generalizes results on the rigidity of Bergman submodules over $\mathbb{D}^n$ generated by polynomials to the infinite-variable setting. In general, traditional methods are not valid because the ring $\mathcal{P}_\infty$ is no longer Notherian. Proofs in this paper yield a more general result, Theorem \ref{thm302}, which covers almost all previously known results in finite-variable setting.
For example, for the Hardy space, our proof significantly simplifies the existing one in \cite{Guo1};  while for the Bergman space, a present proof is more novel than one in \cite{CG, Pu}.

This paper is organized as follows. Section 2 is dedicated to an introduction to basic notations for $A_{\bm{\beta}}^2$ and some preparatory results. In Section 3, we completely characterize the unitary equivalence of submodules of $A_{\bm{\beta}}^2$ generated by polynomials. Section 4 is concerned with some results on principle submodules of $A_{\bm{\beta}}^2$.

\section{Preliminaries}

This section consists of two parts. In the first part, we introduce some basic notations for the weighted Bergman space $A_{\bm{\beta}}^2$ in infinitely many variables. The second part is devoted to some preparatory results.

\subsection{Some basic notations for the weighted Bergman space in infinitely many variables}

For a sequence $\bm{\beta}=\{\beta_{n}\}_{n\in\mathbb{N}}$ in $[-1,\infty)$, as defined in Introduction, the weighted Bergman space $A_{\bm{\beta}}^2$ is exactly the tensor product of $\{A_{\beta_{n}}^{2}\}_{n\in\mathbb{N}}$ with stabilizing sequence $\{1\}_{n\in\mathbb{N}}$ \cite{Ber}. That is to say, $A_{\bm{\beta}}^2$ is the completion of $\mathcal{P}_{\infty}$ with respect to the inner product $\langle\cdot,\cdot\rangle$ determined by
$$\langle\zeta^{\gamma},\zeta^{\delta}\rangle=\prod_{n=1}^{\infty}\langle\zeta_{n}^{\gamma_{n}},\zeta_{n}^{\delta_{n}}\rangle_{A_{\beta_{n}}^{2}},
\quad\gamma,\delta\in\mathbb{Z}_{+}^{\infty},$$
where $\langle\cdot,\cdot\rangle_{A_{\beta_{n}}^{2}}$ denotes the inner product in $A_{\beta_{n}}^{2}$. We define the weighted Hilbert multidisk $\mathbb{D}_{2,\bm{\beta}}^{\infty}$ as follows:
$$\mathbb{D}_{2,\bm{\beta}}^{\infty}=\left\{\zeta=(\zeta_1,\zeta_2,\cdots): \sum\limits_{n=1}^{\infty}(\beta_{n}+2)|\zeta_n|^2<\infty\;\text{and for each}\;n\in\mathbb{N},\;|\zeta_n|<1\right\}.$$
In particular, when $\bm{\beta}=\{-1\}_{n\in\mathbb{N}}$, $\mathbb{D}_{2,\bm{\beta}}^{\infty}$ is the standard Hilbert multidisk
$$\mathbb{D}_{2}^{\infty}=\left\{\zeta=(\zeta_1,\zeta_2,\cdots)\in l^{2}: \;\text{for each}\;n\in\mathbb{N},\;|\zeta_n|<1\right\}$$
(see \cite{Ni}). It is easy to check that all functions in $A_{\bm{\beta}}^2$ are holomorphic in $\mathbb{D}_{2,\bm{\beta}}^{\infty}$ (we refer the readers to \cite{Di} for the definition of holomorphic functions in a domain in some Banach space), and $A_{\bm{\beta}}^2$ is a reproducing kernel Hilbert space with the kernel
$$K_{\lambda}(\zeta)=\prod\limits_{n=1}^{\infty}\frac{1}{(1-\overline{\lambda}_{n}\zeta_{n})^{\beta_{n}+2}},\quad\lambda\in\mathbb{D}_{2,\bm{\beta}}^{\infty}.$$

Following \cite{Ni}, for a nonempty subset $S\subset\mathbb{N}$, we define $\mathbb{D}_{2}^{S}$ to be sequences in $\mathbb{D}_{2}^{\infty}$ which are supported in $S$, that is,
$$\mathbb{D}_{2}^{S}=\{\zeta\in\mathbb{D}_{2}^{\infty}:\;\text{for each}\;n\notin S,\;\zeta_{n}=0\}.$$
When $S$ is finite, $\mathbb{D}_{2}^{S}$ is abbreviated as $\mathbb{D}^{S}$. For a nonempty subset $S$ of $\mathbb{N}$, we define $\mathcal{P}_S$ to be the polynomial ring generated by $\{\zeta_{j}\}_{j\in S}$, and $\mathcal{P}_{\varnothing}$ to be the complex field $\mathbb{C}$. Following the definition of $A_{\bm{\beta}}^2$, the tensor product $A_{\bm{\beta},S}^2=\bigotimes_{n\in S}A_{\beta_{n}}^{2}$ is defined similarly. It is a simple matter to see that $A_{\bm{\beta},S}^2$ is exactly the closure of $\mathcal{P}_S$ in $A_{\bm{\beta}}^2$. If $\beta_{n}=-1$ for every $n\in S$, then $A_{\bm{\beta},S}^2$ is the Hardy space $H^{2}(\mathbb{D}_{2}^{S})$ (see \cite[pp.1613]{Ni} for a definition). When $S=\{1,2,\cdots,n\}$, we write $A_{\bm{\beta},n}^2$ instead of $A_{\bm{\beta},S}^2$. In this situation, $A_{\bm{\beta},n}^2$ is exactly identified with a function space over the polydisk $\mathbb{D}^n$.

Fix a sequence $\bm{\beta}=\{\beta_{n}\}_{n\in\mathbb{N}}$ in $[-1,\infty)$, put $H=\{n\in\mathbb{N}: \beta_{n}=-1\}$ and $B=\{n\in\mathbb{N}: \beta_{n}\neq-1\}$.
Then  $A_{\bm{\beta}}^2$ has a canonical decomposition
$$A_{\bm{\beta}}^2=H^{2}(\mathbb{D}_{2}^{H})\otimes A_{\bm{\beta},B}^2.$$
Set $\mathbb{Z}_{+}^{B}$ to be those elements in  $\mathbb{Z}_{+}^{\infty}$ supported in $B$, that is,
$$\mathbb{Z}_{+}^{B}=\{\alpha\in\mathbb{Z}_{+}^{\infty}: \alpha_{n}=0\;\text{for all}\;n\notin B\}.$$
This implies that  every function $F\in A_{\bm{\beta}}^2$ can be expanded uniquely as follows:
$$F=\sum_{\alpha\in\mathbb{Z}_{+}^{B}}F_{\alpha}\zeta^{\alpha},$$
where $\{F_{\alpha}: \alpha\in\mathbb{Z}_{+}^{B}\}\subset H^{2}(\mathbb{D}_{2}^{H})$. For each $\alpha\in\mathbb{Z}_{+}^{B}$, we define a linear operator $C_{\alpha}$ on $A_{\bm{\beta}}^2$ by putting $C_{\alpha}F=F_{\alpha}$. An easy computation shows that each $C_{\alpha} (\alpha\in\mathbb{Z}_{+}^{B})$ is bounded on $A_{\bm{\beta}}^2$.

For every subset $S$ of $\mathbb{N}$, let $E_{S}$ be the orthogonal projection from $A_{\bm{\beta}}^2$ onto $A_{\bm{\beta},S}^2$. It is easy to check that for $F\in A_{\bm{\beta}}^2$,
$$(E_{S}F)(\zeta)=F(\hat{\zeta}),\quad\zeta\in\mathbb{D}_{2,\bm{\beta}}^{\infty},$$
where $\hat{\zeta}\in\mathbb{D}_{2,\bm{\beta}}^{\infty}$ such that
$$\hat{\zeta}_{n}=\begin{cases} \zeta_{n} & n\in S, \\
0 & n\notin S. \end{cases}$$
Therefore $E_{S}$ is multiplicative on $A_{\bm{\beta}}^2$ in the sense that $E_{S}(FG)=E_{S}F\cdot E_{S}G$ if $F,G$ and $FG\in A_{\bm{\beta}}^2$.

When $S=\{1,2,\cdots,n\}$, we write $E_{n}$ in stead of $E_{S}$. Then for each $F\in A_{\bm{\beta}}^2$, the sequence $\{\|E_{n}F\|\}_{n=1}^{\infty}$ is ascending, and $\|F-E_{n}F\|\rightarrow0\ (n\rightarrow\infty)$. Set $dA_{-1}=\frac{d\theta}{2\pi}$, the normalized Lebesgue measure on the unit circle $\mathbb{T}$. An important observation is  that for each $F\in A_{\bm{\beta}}^2$,
$$\|F\|^{2}=\sup\limits_{n\geq1}\sup\limits_{0<r<1}\int_{\Xi_{n}}|(E_{n}F)(r\zeta_{1},\cdots,r\zeta_{n})|^{2}dA_{\beta_{1}}(\zeta_{1})\cdots dA_{\beta_{n}}(\zeta_{n}),$$
where
$$\Xi_{n}=\left\{\zeta=(\zeta_{1},\cdots,\zeta_{n})\in\overline{\mathbb{D}}^{n}:\;\text{when}\;k\in H,\;\zeta_{k}\in\mathbb{T}\;\text{; otherwise,}\;\zeta_{k}\in\overline{\mathbb{D}}\right\},$$
and $\overline{\mathbb{D}}$ denotes the closed unit disk.

\subsection{Some preparatory results}

In this subsection, we will introduce some preparatory results. First, let us introduce some notations. For $S\subset\mathbb{N}$ and a subset $\mathcal{F}$ of $A_{\bm{\beta},S}^2$, let $[\mathcal{F}]_{S}$ be the submodule of $A_{\bm{\beta},S}^2$ generated by $\mathcal{F}$. When $S=\{1,2,\cdots,n\}$, we denote $[\mathcal{F}]_{S}$ by $[\mathcal{F}]_{n}$; when $S=\mathbb{N}$, we denote $[\mathcal{F}]_{S}$ by $[\mathcal{F}]$.

Suppose $\varphi\in A_{\bm{\beta},n}^2$. It follows from \cite[Theorem 2.3]{Guo1} that each function $f$ in $[\varphi]_n$ admits a  factorization $f=\varphi g$, where $g$ is holomorphic in $\mathbb{D}^n$. Conversely, if a holomorphic function $g$ is such that $\varphi g\in  A_{\bm{\beta},n}^2$,  we want to know  whether $\varphi g$ belongs to $[\varphi]_n$. The following theorem gives a positive answer in the case that  $\varphi$ is holomorphic in some neighborhood of $\overline{\mathbb{D}}^n$.

\begin{thm}\label{thm201}
Suppose $\varphi$ is holomorphic in some neighborhood of $\overline{\mathbb{D}}^n$. Then
$$[\varphi]_n=\left\{\varphi f\in A_{\bm{\beta},n}^2:f\ \mathrm{is}\ \mathrm{holomorphic}\ \mathrm{in}\ \mathbb{D}^n\right\}.$$
\end{thm}

For each $r\in(0,1)$ and a holomorphic function $f$ in $\mathbb{D}^n$, we define $f_{r}(\zeta)=f(r\zeta)$, $\zeta\in\mathbb{D}^{n}$. To prove Theorem \ref{thm201}, we need the following lemma. It can be obtained by a similar proof as in \cite[Lemma 3.5]{DG}.

\begin{lem}\label{thm203}
Let $f$ be holomorphic in some neighborhood of $\overline{\mathbb{D}}^n$ with $f(0)\neq0$. If $g$ is  holomorphic  in $\mathbb{D}^n$ such that $fg\in A_{\bm{\beta},n}^2$, then $\{fg_r\}$ converges to $fg$ weakly in $A_{\bm{\beta},n}^2$ as $r\rightarrow1^-$.
\end{lem}

We now give the proof of Theorem \ref{thm201}.

\vskip2mm

\noindent\textbf{Proof of Theorem \ref{thm201}.} As mentioned above, we have obtained the inclusion in one direction. The opposite inclusion under the assumption $\varphi(0)\neq0$ can be viewed from Lemma \ref{thm203}. In fact, suppose that $\varphi f\in A_{\bm{\beta},n}^2$ for some holomorphic function $f$ in $\mathbb{D}^n$. By Lemma \ref{thm203}, $\{\varphi f_r\}$ converges to $\varphi f$ weakly in $A_{\bm{\beta},n}^2$ as $r\rightarrow1^-$, and thus $\varphi f\in[\varphi]_n$.

For the general case, we may assume without loss of generality that $\varphi\not \equiv0$. Choose a point $\xi=(\xi_1,\cdots,\xi_n)\in\mathbb{D}^n$ with $\varphi(\xi)\neq0$. It is easy to see that the image of $\varphi$ under the standard unitary operator
$$(Wf)(\zeta)=f\left(\frac{\xi_1-\zeta_1}{1-\overline{\xi_1}\zeta_1},\cdots,\frac{\xi_n-\zeta_n}
{1-\overline{\xi_n}\zeta_n}\right)
\cdot\prod\limits_{k=1}^{n}\left(\frac{\sqrt{1-|\xi_{k}|^{2}}}{1-\bar{\xi}_{k}\zeta_{k}}\right)^{2+\beta_{k}},\quad f\in A_{\bm{\beta},n}^2$$
is also holomorphic in some neighborhood of $\overline{\mathbb{D}}^n$, and $(W\varphi)(0)\neq0$. By the proved conclusion,
$$[W\varphi]_n=
\left\{(W\varphi)f\in A_{\bm{\beta},n}^2:f\ \mathrm{is}\ \mathrm{holomorphic}\ \mathrm{in}\ \mathbb{D}^n\right\}.$$
On the other hand, since $W$ maps $\varphi H^\infty(\mathbb{D}^n)$ onto $(W\varphi)H^\infty(\mathbb{D}^n)$, it follows that $W[\varphi]_n=[W\varphi]_n$. Since $W$ is adjoint,
$$[\varphi]_n=W[W\varphi]_n=\left\{\varphi f\in A_{\bm{\beta},n}^2:f\ \mathrm{is}\ \mathrm{holomorphic}\ \mathrm{in}\ \mathbb{D}^n\right\}.$$
The proof is complete. $\hfill\square$

\vskip2mm

The following lemma will be used to prove Lemma \ref{thm309} and Lemma \ref{thm310} in Section 3.

\begin{lem}\label{thm204}
Suppose $-1<\beta<\infty$, and $\{b_{n}\}_{n\geq0}$ is a sequence of real numbers. If
\begin{equation}\label{equ201}
\sum\limits_{n=0}^{\infty}\frac{|b_{n}|n!}{\Gamma(n+\beta+2)}<\infty,
\end{equation}
and for each nonnegative integer $k$,
$$\sum\limits_{n=0}^{\infty}\frac{b_{n}(n+k)!}{\Gamma(n+k+\beta+2)}=0,$$
then $b_{n}=0\ (n\geq0)$.
\end{lem}

\begin{proof}
Write $\psi(t)=\sum_{n=0}^{\infty}b_{n}t^{n}$. For each $t\in[0,1)$, it follows that $t^n\leq\frac{n!}{\Gamma(n+\beta+2)}$ for $n\in\mathbb{N}$ sufficiently large. Combining this with (\ref{equ201}) yields the power series $\sum_{n=0}^{\infty}b_{n}t^{n}$ converges, and thus $\psi$ is continuous on $[0,1)$.

Put $\varphi(t)=(1-t)^{\beta}\psi(t)$. Then $\varphi$ is continuous on $[0,1)$, and it follows from (\ref{equ201}) that $\varphi$ is integrable over $[0,1]$. Since for each nonnegative integer $k$,
$$\int_{0}^{1}t^{k}\varphi(t)dt=\sum\limits_{n=0}^{\infty}b_{n}\frac{(n+k)!\Gamma(\beta+1)}{\Gamma(n+k+\beta+2)}=0,$$
we conclude that for each real-valued continuous function $f$ on $[0,1]$,
$$\int_{0}^{1}f(t)\varphi(t)dt=0.$$
It follows from the Riesz representation theorem that $\varphi\equiv0$ on $[0,1)$. Therefore $\psi\equiv0$ on $[0,1)$, which leads to $b_{n}=0\ (n\geq0)$.
\end{proof}

\section{The unitary equivalence of submodules generated by polynomials}

This section is devoted to the unitary equivalence of submodules of $A_{\bm{\beta}}^2$ generated by polynomials.

To begin, some notions related to polynomials are needed. Suppose $S$ is a subset of $\mathbb{N}$. For each $n\in\mathbb{N}$, we write $S(n)=S\cap\{1,2,\cdots,n\}$. As in the case of finitely many variables, every nonzero ideal $\mathcal{I}$ of $\mathcal{P}_S$ has a Beurling form $\mathcal{I}=p\mathcal{L}$, where $p$ is the greatest common divisor of $\mathcal{I}$, and $\mathcal{L}$ is an ideal of $\mathcal{P}_S$ without nontrivial common divisor. Indeed, for  $n\in\mathbb{N}$ sufficiently large, $\mathcal{I}\cap\mathcal{P}_{S(n)}$ is a nonzero ideal of $\mathcal{P}_{S(n)}$, and thus has its greatest common divisor $p_n$. This means that $p_{n+1}|p_n$, and thus the degrees of $p_n$ are descending. This shows that there exists a positive integer $m$ such that $\deg p_n=\deg p_m$ for any $n>m$. Thus $p_m$ is the greatest common divisor of $\mathcal{I}$ as desired.

Now we define an equivalence relation on polynomial ring $\mathcal{P}_\infty$. Let $\mathbb{D}^\infty=\mathbb{D}\times\mathbb{D}\times\cdots$ and $\mathbb{T}^\infty=\mathbb{T}\times\mathbb{T}\times\cdots$ be the cartesian product of countably infinitely many open unit disks $\mathbb{D}$ and unit circles $\mathbb{T}$, respectively. For each $F\in A_{\bm{\beta}}^2$, we write
$$\mathbb{N}_F=\left\{n\in\mathbb{N}: \frac{\partial F}{\partial\zeta_{n}}\not\equiv0\right\}.$$
Following \cite{Guo1}, say two polynomials $p,q\in\mathcal{P}_\infty$ are modulus equivalent, if there are two polynomials $r,s\in\mathcal{P}_{\infty}$ without zero point in $\mathbb{D}^\infty$, such that $|pr|=|qs|$ on $\mathbb{T}^\infty$. In this case, write $p\asymp q$. It is routine to check that the binary relation $\asymp$ on $\mathcal{P}_\infty$ is an equivalence relation. By Hurwitz's theorem, if a polynomial $u\in\mathcal{P}_{n}$ is zero-free in $\mathbb{D}^n$, then there is a $\zeta\in\mathbb{T}$ such that $u(\cdot,\zeta)\in\mathcal{P}_{n-1}$ is zero-free in $\mathbb{D}^{n-1}$. Therefore, when $p,q\in\mathcal{P}_\infty$ are modulus equivalent, the corresponding $r,s$ can be chosen from $\mathcal{P}_{\mathbb{N}_{p}\cup\mathbb{N}_{q}}$.

Our main result in this paper  is as follows, which gives a complete classification under unitary equivalence for submodules of
$A_{\bm{\beta}}^2$ generated by polynomials. Here, we mention a simple fact that if  $M$ is a submodule of $A_{\bm{\beta}}^2$, then $M\cap\mathcal{P}_{\infty}$ is an ideal of $\mathcal{P}_{\infty}$, see \cite{DGH}.

\begin{thm}\label{thm301}
Suppose that $\bm{\beta}=\{\beta_{n}\}_{n\in\mathbb{N}}$ is a sequence in $[-1,\infty)$ and $H=\{n\in\mathbb{N}: \beta_{n}=-1\}$. Let $M$, $N$ be two nonzero submodules of $A_{\bm{\beta}}^2$ generated by polynomials, and let
$M\cap\mathcal{P}_{\infty}=p\mathcal{K}$, $N\cap\mathcal{P}_{\infty}=q\mathcal{L}$ be the Beurling forms of ideals $M\cap\mathcal{P}_{\infty}$ and $N\cap\mathcal{P}_{\infty}$, respectively. Then the following are equivalent:

(1) $M$ is unitarily equivalent to $N$;

(2) $\mathcal{K}=\mathcal{L}$ and the polynomials $p,q$ admit factorizations $p=r\varphi$, $q=s\varphi$, where $r,s\in\mathcal{P}_{H}$, $\varphi\in\mathcal{P}_{\infty}$, and $r\asymp s$;

(3) there exist an ideal $\mathcal{G}$ of $\mathcal{P}_\infty$, and polynomials $\tilde{p},\tilde{q}\in\mathcal{P}_{H}$ with $|\tilde{p}|=|\tilde{q}|$ on $\mathbb{T}^\infty$ such that $M$ and $N$ are the closures of $\tilde{p}\mathcal{G}$ and $\tilde{q}\mathcal{G}$, respectively.
\end{thm}

More generally, we have the following theorem.

\begin{thm}\label{thm302}
Suppose that $\bm{\beta}=\{\beta_{n}\}_{n\in\mathbb{N}}$ is a sequence in $[-1,\infty)$, $S$ is a nonempty subset of $\mathbb{N}$, and $H=\{n\in\mathbb{N}: \beta_{n}=-1\}$. Let $M$, $N$ be two nonzero submodules of $A_{\bm{\beta},S}^2$ generated by some polynomials in $\mathcal{P}_{S}$, and let $M\cap\mathcal{P}_{S}=p\mathcal{K}$, $N\cap\mathcal{P}_{S}=q\mathcal{L}$ be the Beurling forms of ideals $M\cap\mathcal{P}_{S}$ and $N\cap\mathcal{P}_{S}$, respectively. Then the following are equivalent:

(1) $M$ is unitarily equivalent to $N$;

(2) $\mathcal{K}=\mathcal{L}$ and the polynomials $p,q$ admit factorizations $p=r\varphi$, $q=s\varphi$, where $r,s\in\mathcal{P}_{S\cap H}$, $\varphi\in\mathcal{P}_{S}$, and $r\asymp s$;

(3) there exist an ideal $\mathcal{G}$ of $\mathcal{P}_S$, and polynomials $\tilde{p},\tilde{q}\in\mathcal{P}_{S\cap H}$ with $|\tilde{p}|=|\tilde{q}|$ on $\mathbb{T}^\infty$ such that $M$ and $N$ are the closures of $\tilde{p}\mathcal{G}$ and $\tilde{q}\mathcal{G}$, respectively.
\end{thm}

By taking $S=\{1,2,\cdots,n\}$ in Theorem \ref{thm302}, we obtain  the  version of Theorem \ref{thm301} in  finite-variable setting. In this paper, we only give the proof of Theorem \ref{thm301}. Our methods are also valid to prove Theorem \ref{thm302}.

Before proving Theorem \ref{thm301}, we give some applications. Corollary \ref{thm303} below generalizes \cite[Corollary 3]{ACD}. It shows the rigidity of finite codimensional submodules of $A_{\bm{\beta}}^2$.

\begin{cor}\label{thm303}
If $M$ and $N$ are submodules of finite codimension in $A_{\bm{\beta}}^2$, then they are unitarily equivalent only when they coincide.
\end{cor}

\begin{proof}
A similar proof as in \cite[Theorem 3.2]{DGH} (or see \cite[Proposition 2.4]{DPSY}) shows that submodules $M$, $N$ are generated by $M\cap\mathcal{P}_\infty$ and $N\cap\mathcal{P}_\infty$, respectively, and the codimensions of both ideals  $M\cap\mathcal{P}_\infty$ and  $N\cap\mathcal{P}_\infty$   in  $\mathcal{P}_\infty$ are  finite.  This implies that each of this two ideals is of trivial greatest common divisor. By Theorem \ref{thm301}, $M\cap\mathcal{P}_\infty=N\cap\mathcal{P}_\infty$, and hence $M=N$.
\end{proof}

Applying Theorem \ref{thm301} to the Hardy module $H^{2}(\mathbb{D}_{2}^{\infty})$, we obtain the following.

\begin{cor}\label{thm304}
Let $M$, $N$ be two nonzero submodules of $H^2(\mathbb{D}_2^\infty)$ generated by polynomials, and let $M\cap\mathcal{P}_\infty=p\mathcal{K}$, $N\cap\mathcal{P}_\infty=q\mathcal{L}$ be the Beurling forms of ideals $M\cap\mathcal{P}_\infty$ and $N\cap\mathcal{P}_\infty$, respectively. Then $M$ is unitarily equivalent to $N$ if and only if $p\asymp q$ and $\mathcal{K}=\mathcal{L}$.
\end{cor}

Below, we take $\bm{\eta}=\{\eta_{n}\}_{n\in\mathbb{N}}$ to be a sequence in $(-1,\infty)$. In this case, the corresponding $H=\varnothing$. Applying Theorem \ref{thm301} to polynomially generated submodules of $A_{\bm{\eta}}^{2}$, we have the following corollary which shows rigidity of such submodules.

\begin{cor}\label{thm305}
Let $\bm{\eta}=\{\eta_{n}\}_{n\in\mathbb{N}}$ be a sequence in $(-1,\infty)$, and let $M$, $N$ be two submodules of $A_{\bm{\eta}}^{2}$ generated by polynomials. Then $M$ is unitarily equivalent to $N$ only if $M=N$.
\end{cor}

The remainder of this section will be devoted to the proof of Theorem \ref{thm301}. First, we recall again for a sequence $\bm{\beta}=\{\beta_{n}\}_{n\in\mathbb{N}}$ in $[-1,\infty)$, $H=\{n\in\mathbb{N}: \beta_{n}=-1\}$ and
$B=\{n\in\mathbb{N}: \beta_{n}\neq-1\}$. In what follows, we introduce a proposition. It will play an important role in the proof of Theorem \ref{thm301}.

\begin{prop}\label{thm306}
Let $M$, $N$ be two nonzero submodules of $A_{\bm{\beta}}^2$ generated by some polynomials in $\mathcal{P}_{H}$, and $p$, $q$ denote the greatest common divisors of ideals $M\cap\mathcal{P}_{\infty}$ and $N\cap\mathcal{P}_{\infty}$, respectively. If $M$ is unitarily equivalent to $N$ via a unitary module map $U$, then $U$ can be extended to a unitary module map $\widetilde{U}:[p]\rightarrow[q]$.
\end{prop}

To give the proof of Proposition \ref{thm306}, we need two lemmas.

\begin{lem}\label{thm307}
Suppose $S$ is a nonempty subset of $\mathbb{N}$. Let $M$, $N$ be submodules of $A_{\bm{\beta}}^2$ generated by some polynomials in $\mathcal{P}_{S}$. If $M$ is unitarily equivalent to $N$ via a unitary module map $U$, then there exists a positive integer $K$, such that for each $k\geq K$,
$$U\left(M\cap A_{\bm{\beta},S(k)}^{2}\right)=N\cap A_{\bm{\beta},S(k)}^{2}.$$
\end{lem}

\begin{proof}
Take a positive integer $K$ such that both $M\cap\mathcal{P}_{S(K)}$ and $N\cap\mathcal{P}_{S(K)}$ are nonzero. It is easy to see that both $M\cap\mathcal{P}_{S(K)}$ and $N\cap\mathcal{P}_{S(K)}$ are ideals of $\mathcal{P}_{S(K)}$. It will be shown if $k\geq K$ and $F\in M\cap A_{\bm{\beta},S(k)}^{2}$, then $UF\in A_{\bm{\beta},S(k)}^{2}$.

Take a nonzero polynomial $r\in N\cap\mathcal{P}_{S(k)}$. Then
\begin{equation}\label{equ301}
E_{S(k)}U(rF)=E_{S(k)}(r\cdot UF)=r\cdot E_{S(k)}UF\in N.
\end{equation}
By an approximation argument, for each $G\in A_{\bm{\beta}}^{2}$,
\begin{equation}\label{equ3021}
U^{-1}(rG)=G\cdot U^{-1}r.
\end{equation}
Then $U^{-1}r$ is a multiplier of $A_{\bm{\beta}}^{2}$. Taking $G=E_{S(k)}UF$ in (\ref{equ3021}) we then obtain
\begin{equation}\label{equ302}
U^{-1}(r\cdot E_{S(k)}UF)=E_{S(k)}UF\cdot U^{-1}r.
\end{equation}
Taking $G=UF$ in (\ref{equ3021}) we then obtain
\begin{equation}\label{equ303}
UF\cdot U^{-1}r=U^{-1}(r\cdot UF)=U^{-1}(U(rF))=rF.
\end{equation}
Therefore
\begin{equation}\label{equ304}
U^{-1}E_{S(k)}U(rF)=U^{-1}(r\cdot E_{S(k)}UF)=E_{S(k)}UF\cdot U^{-1}r,
\end{equation}
where the first equality follows from (\ref{equ301}) and the second follows from (\ref{equ302}). On the other hand, since $E_{S(k)}$ is multiplicative and idempotent on $A_{\bm{\beta}}^2$,
$$E_{S(k)}(E_{S(k)}UF\cdot U^{-1}r)=E_{S(k)}^{2}UF\cdot E_{S(k)}U^{-1}r=E_{S(k)}(UF\cdot U^{-1}r).$$
Then it follows from (\ref{equ303}) that
$$E_{S(k)}(E_{S(k)}UF\cdot U^{-1}r)=E_{S(k)}(UF\cdot U^{-1}r)=E_{S(k)}(rF)=rF.$$
Combining (\ref{equ304}) with this equality gives
$$E_{S(k)}U^{-1}E_{S(k)}U(rF)=E_{S(k)}(E_{S(k)}UF\cdot U^{-1}r)=rF,$$
and thus
$$\|U(rF)\|=\|rF\|=\|E_{S(k)}U^{-1}E_{S(k)}U(rF)\|\leq\|E_{S(k)}U(rF)\|\leq\|U(rF)\|.$$
This implies
$$\|U(rF)\|=\|E_{S(k)}U(rF)\|,$$
forcing
$$U(rF)=E_{S(k)}U(rF),$$
and hence
$$r\cdot UF=r\cdot E_{S(k)}UF.$$
Then
$$UF=E_{S(k)}UF\in A_{\bm{\beta},S(k)}^{2},$$
which implies
$$U\left(M\cap A_{\bm{\beta},S(k)}^{2}\right)\subset N\cap A_{\bm{\beta},S(k)}^{2}.$$
Applying a similar argument yields the opposite inclusion, which completes the proof.
\end{proof}

\begin{lem}\label{thm308}
If $\mathcal{E}$ is a nonzero ideal of $\mathcal{P}_{H}$, then $[\mathcal{E}]\cap\mathcal{P}_{\infty}$ and $[\mathcal{E}]\cap\mathcal{P}_{H}$ have the same greatest common divisor, except a constant factor.
\end{lem}

\begin{proof}
Let $\varphi_{1}$, $\varphi_{2}$ be the greatest common divisors of $[\mathcal{E}]\cap\mathcal{P}_{\infty}$ and $[\mathcal{E}]\cap\mathcal{P}_{H}$, respectively. It is clear that $\varphi_{1}|\varphi_{2}$.

On the other hand, it is easy to verify that $[\mathcal{E}]=[\mathcal{E}]_{H}\otimes A_{\bm{\beta},B}^{2}$. Then for each polynomial $r\in[\mathcal{E}]\cap\mathcal{P}_{\infty}$, there exists a positive integer $N$ and polynomials $\{r_{k}\}_{k=1}^{N}\subset[\mathcal{E}]\cap\mathcal{P}_{H}$, $\{p_{k}\}_{k=1}^{N}\subset\mathcal{P}_{B}$, such that $r=\sum_{k=1}^{N}r_{k}p_{k}$. This implies $\varphi_{2}|r$, and thus $\varphi_{2}|\varphi_{1}$. Therefore $\varphi_1=c\varphi_2$ for some constant $c$.
\end{proof}

Before proving Proposition \ref{thm306}, we recall some notions related to the Hardy space $H^{2}(\mathbb{D}_{2}^{H})$. Let
$$\mathbb{T}^{H}=\{\zeta\in\mathbb{T}^\infty:\;\text{for each}\;n\notin H,\;\zeta_{n}=1\}.$$
As a subgroup of $\mathbb{T}^\infty$, it is compact. The Haar measure on $\mathbb{T}^{H}$ will be denoted by $\sigma_{H}$. The Hardy space over $\mathbb{T}^{H}$, denoted by $H^{2}(\mathbb{T}^{H})$, is the closure of $\mathcal{P}_{H}$ in the space $L^{2}(\mathbb{T}^{H},\sigma_{H})$. With each function $F\in H^{2}(\mathbb{D}_{2}^{H})$ is associated its boundary-value function $F^{*}\in H^{2}(\mathbb{T}^{H})$, and
\begin{equation}\label{equ315}
\|F\|^2=\int_{\mathbb{T}^{H}}|F^{*}|^{2}d\sigma_{H}.
\end{equation}
Moreover, $H^{2}(\mathbb{D}_{2}^{H})$ is isometrically isomorphic to $H^{2}(\mathbb{T}^{H})$, and the isomorphism $H^{2}(\mathbb{T}^{H})=H^{2}(\mathbb{D}_{2}^{H})$ is realized by taking Poisson integrals (see \cite{Ru} for the finite-variable case; see \cite{AOS, CG1} for the infinite-variable case). Hence we may identify a function in $H^{2}(\mathbb{D}_{2}^{H})$ with its boundary-value function in $H^{2}(\mathbb{T}^{H})$ if we need.

To prove Proposition \ref{thm306}, we also need the following reasoning. Suppose $S$ is a subset of $\mathbb{N}$. For $p\in\mathcal{P}_{S}$, since $p\cdot\mathcal{P}_{\infty}$ is dense in $[p]$, it follows that $p\cdot\mathcal{P}_{S}=E_{S}(p\cdot\mathcal{P}_{\infty})$ is dense in $E_{S}[p]$, and thus $[p]_S=E_{S}[p]$. On the other hand,
$$[p]_{S}\subset[p]\cap A_{\bm{\beta},S}^{2}\subset E_{S}[p].$$
Thus we have proved
\begin{equation}\label{equ305}
[p]_{S}=[p]\cap A_{\bm{\beta},S}^{2}=E_{S}[p].
\end{equation}

We also recall that every function $F\in A_{\bm{\beta}}^2$ can be expanded uniquely as follows:
\begin{equation}\label{equ335}
F=\sum_{\alpha\in\mathbb{Z}_{+}^{B}}C_{\alpha}F\cdot\zeta^{\alpha},
\end{equation}
where $\{C_{\alpha}F: \alpha\in\mathbb{Z}_{+}^{B}\}\subset H^{2}(\mathbb{D}_{2}^{H})$.

We now prove Proposition \ref{thm306}.

\vskip2mm

\noindent\textbf{Proof of Proposition \ref{thm306}.} If $H=\varnothing$, then $M=N=A_{\bm{\beta}}^2$, and the conclusion is trivial. Below we suppose that $H$ is nonempty. It is clear that both $M\cap\mathcal{P}_{H}$ and $N\cap\mathcal{P}_{H}$ are ideals of $\mathcal{P}_{H}$. By Lemma \ref{thm308}, $p,q\in\mathcal{P}_{H}$ are the greatest common divisors of $M\cap\mathcal{P}_{H}$ and $N\cap\mathcal{P}_{H}$, respectively. So we assume
$M\cap\mathcal{P}_{H}=p\mathcal{K}$ and $N\cap\mathcal{P}_{H}=q\mathcal{L}$ are the Beurling forms of $M\cap\mathcal{P}_{H}$ and
$N\cap\mathcal{P}_{H}$, respectively.

Taking $S=H$ in Lemma \ref{thm307} yields that for $n$ sufficiently large,
\begin{equation}\label{equ306}
U\left(M\cap H^{2}(\mathbb{D}^{H(n)})\right)=N\cap H^{2}(\mathbb{D}^{H(n)}).
\end{equation}
Moreover, we suppose that for such $n$, $\mathcal{K}\cap\mathcal{P}_{H(n)}$ only have trivial common divisors, and $p,q$ are the greatest common divisors of $M\cap\mathcal{P}_{H(n)}$ and $N\cap\mathcal{P}_{H(n)}$, respectively.

Without loss of generality we may assume $H(n)=\{1,2,\cdots,n\}$. It follows from \cite[Lemma 1]{ACD} that there is a function $\eta$ on $\mathbb{T}^{n}$ with modulus $1$, such that for $f\in M\cap H^{2}(\mathbb{D}^{n})$,
\begin{equation}\label{equ307}
Uf=\eta f
\end{equation}
on $\mathbb{T}^{n}$. Then \cite[Lemma 3.5]{Guo1} immediately yields $\eta p\in H^\infty(\mathbb{T}^{n})$. Put $\tilde{p}=\eta p$ and $\tilde{q}=\bar{\eta}q$. It will be shown that $\tilde{p}\in[q]_n$ and $\tilde{q}\in[p]_n$.

For $u\in\mathcal{K}\cap\mathcal{P}_{n}$, we have $pu\in M\cap H^{2}(\mathbb{D}^{n})$. Then by (\ref{equ307}),
$$U(pu)=\eta pu=\tilde{p}u.$$
Combining this with (\ref{equ305}) and (\ref{equ306}) shows that
$$\tilde{p}u\in N\cap H^{2}(\mathbb{D}^{n})\subset [q]\cap H^{2}(\mathbb{D}^{n})=[q]_{n}.$$
It follows from Theorem \ref{thm201} that for each $u\in\mathcal{K}\cap\mathcal{P}_{n}$, there corresponds a holomorphic function $h_u$ in $\mathbb{D}^{n}$ such that $\tilde{p}u=qh_u$. It is known that the polynomial ring $\mathcal{P}_{n}$ is Noetherian, and hence $\mathcal{K}\cap\mathcal{P}_{n}$, as an ideal of $\mathcal{P}_{n}$, is finitely generated \cite{AM}. Suppose that $\mathcal{K}\cap\mathcal{P}_{n}$ is generated by $u_1,\cdots,u_m$,
and put
$$Z(\mathcal{K}\cap\mathcal{P}_{n})=\{\zeta\in\mathbb{D}^{n}: u(\zeta)=0\   \mathrm{for}\ \mathrm{all}\ u\in\mathcal{K}\cap\mathcal{P}_{n}\},$$
the zero set of $\mathcal{K}\cap\mathcal{P}_{n}$ in $\mathbb{D}^{n}$. Noticing that
$$\frac{\tilde{p}}{q}=\frac{h_{u_1}}{u_1}=\cdots=\frac{h_{u_m}}{u_m},$$
we see $\frac{\tilde{p}}{q}$ is holomorphic in $\mathbb{D}^{n}\setminus Z(\mathcal{K}\cap\mathcal{P}_{n})$. Since the greatest common divisor of $\mathcal{K}\cap\mathcal{P}_{n}$ is $1$, the set $Z(\mathcal{K}\cap\mathcal{P}_{n})$ has codimension $\geq2$ as an analytic subset of $\mathbb{D}^{n}$ (\cite[Corollary 3.1.12]{CG}), and thus $\frac{\tilde{p}}{q}$ can be extended to a holomorphic function $g$ in $\mathbb{D}^{n}$ by the second removable singularity theorem (\cite[Theorem 7.7]{KK}). Now Theorem \ref{thm201} gives
$\tilde{p}=qg\in[q]_{n}$. Similarly, $\tilde{q}\in[p]_n$.

By (\ref{equ315}) and the conclusion above, it is easily seen that $\eta[p]_H\subset[q]_H$, and for each $h\in[p]_H$, $\|\eta h\|=\|h\|$. Combining the decomposition $[p]=[p]_{H}\otimes A_{\bm{\beta},B}^{2}$ with (\ref{equ335}) gives that $C_{\alpha}F\in[p]_H$ for every $F\in[p]$ and $\alpha\in\mathbb{Z}_{+}^{B}$. Define a linear operator $\widetilde{U}: [p]\rightarrow[q]$ by putting
$$\widetilde{U}F=\sum_{\alpha\in\mathbb{Z}_{+}^{B}}\eta C_{\alpha}F\cdot\zeta^{\alpha},\quad F\in[p].$$
Then for each $F\in[p]$,
$$\|\widetilde{U}F\|^{2}=\sum_{\alpha\in\mathbb{Z}_{+}^{B}}\|\eta C_{\alpha}F\|^{2}\|\zeta^{\alpha}\|^{2}=\sum_{\alpha\in\mathbb{Z}_{+}^{B}}\|C_{\alpha}F\|^{2}\|\zeta^{\alpha}\|^{2}=\|F\|^{2},$$
which implies $\widetilde{U}$ is an isometry. Similarly, we define a linear operator $\widetilde{V}: [q]\rightarrow[p]$ by putting
$$\widetilde{V}G=\sum_{\alpha\in\mathbb{Z}_{+}^{B}}\bar{\eta}C_{\alpha}G\cdot\zeta^{\alpha},\quad G\in[q].$$
A simple verification shows that $\widetilde{U}\widetilde{V}$ is the identity operator on $[q]$. Therefore $\widetilde{U}$ is onto, and hence unitary.

The only thing remaining is to show that $\widetilde{U}$ is an extension of $U$. Indeed, it is clear that $\widetilde{U}|_{p\mathcal{K}}=U|_{p\mathcal{K}}$. Since $M$ is generated by some polynomials in $\mathcal{P}_{H}$, it follows that $M=[M\cap\mathcal{P}_{H}]=[p\mathcal{K}]$, and thus $\widetilde{U}|_{M}=U$. The proof is complete. $\hfill\square$

\vskip2mm

To prove Theorem \ref{thm301}, we also need the following lemma.

\begin{lem}\label{thm309}
Let $M$, $N$ be two submodules of $A_{\bm{\beta}}^2$, and $M$ is unitarily equivalent to $N$ via a unitary module map $U$. If $F\in M$ and $\mathbb{N}_F\cap B$ is a finite set, then for each $\alpha\in\mathbb{Z}_{+}^{B}$, $\|C_{\alpha}F\|=\|C_{\alpha}UF\|$.
\end{lem}

We will prove Lemma \ref{thm309} by virtue of Lemma \ref{thm310} below.

\begin{lem}\label{thm310}
Let $M$, $N$ be two submodules of $A_{\bm{\beta}}^2$, and $M$ is unitarily equivalent to $N$ via a unitary module map $U$. Then for each $F\in M$, $\mathbb{N}_F\cap B=\mathbb{N}_{UF}\cap B$.
\end{lem}

\begin{proof}
We first prove $\mathbb{N}_{UF}\cap B\subset\mathbb{N}_F\cap B$. If $B\setminus\mathbb{N}_F$ is empty, then $\mathbb{N}_{UF}\cap B\subset B=\mathbb{N}_F\cap B$, and the desired conclusion follows. Below, we suppose $B\setminus\mathbb{N}_F$ is nonempty. Pick $j\in B\setminus\mathbb{N}_F$. It is easily seen that
\begin{equation}\label{equ309}
\|\zeta_{j}^{k}\|^{2}\|F\|^{2}=\frac{k!\Gamma(\beta_{j}+2)}{\Gamma(\beta_{j}+k+2)}\|F\|^{2}.
\end{equation}
On the other hand, suppose $(UF)(\zeta)=\sum_{\alpha\in\mathbb{Z}_{+}^{\infty}}c_{\alpha}\zeta^{\alpha}$. Then for each nonnegative integer $k$,
\begin{equation}\label{equ310}
\begin{aligned}\|\zeta_{j}^{k}(UF)\|^{2}&=\sum\limits_{\alpha\in\mathbb{Z}_{+}^\infty}\left[|c_{\alpha}|^{2}\frac{(\alpha_{j}+k)!\Gamma(\beta_{j}+2)}{\Gamma(\alpha_{j}+\beta_{j}+k+2)}\prod\limits_{l\neq j}\frac{\alpha_{l}!\Gamma(\beta_{l}+2)}{\Gamma(\alpha_{l}+\beta_{l}+2)}\right]
\\&=\sum\limits_{n=0}^{\infty}\left[\sum\limits_{\alpha_{j}=n}\left(|c_{\alpha}|^{2}\prod\limits_{l\neq j}\frac{\alpha_{l}!\Gamma(\beta_{l}+2)}{\Gamma(\alpha_{l}+\beta_{l}+2)}\right)\right]\frac{(n+k)!\Gamma(\beta_{j}+2)}{\Gamma(\beta_{j}+n+k+2)}.\end{aligned}
\end{equation}
Note that $j\notin\mathbb{N}_{F}$. It is clear that
$$\|\zeta_{j}^{k}\|\|F\|=\|\zeta_{j}^{k}F\|=\|U(\zeta_{j}^{k}F)\|=\|\zeta_{j}^{k}(UF)\|.$$
It follows from (\ref{equ309}) and (\ref{equ310}) that for each $k\geq0$,
\begin{equation}\label{equ311}
\sum\limits_{n=0}^{\infty}b_{n}\frac{(n+k)!}{\Gamma(\beta_{j}+n+k+2)}=0,
\end{equation}
where
$$b_{0}=\sum\limits_{\alpha_{j}=0}\left(|c_{\alpha}|^{2}\prod\limits_{l\neq j}\frac{\alpha_{l}!\Gamma(\beta_{l}+2)}{\Gamma(\alpha_{l}+\beta_{l}+2)}\right)-\|F\|^{2},$$
and
$$b_{n}=\sum\limits_{\alpha_{j}=n}\left(|c_{\alpha}|^{2}\prod\limits_{l\neq j}\frac{\alpha_{l}!\Gamma(\beta_{l}+2)}{\Gamma(\alpha_{l}+\beta_{l}+2)}\right),\quad n\geq1.$$
To apply Lemma \ref{thm204}, it will be shown that $\{b_{n}\}_{n\geq0}$ is summable. An direct computation shows that when $n\geq1$, the sequence $\left\{\frac{(n+k)!\Gamma(k+\beta_j+2)}{k!\Gamma(n+k+\beta_j+2)}\right\}_{k\geq0}$ is bounded and ascending. Then
$$\sum\limits_{n=1}^{\infty}b_{n}\leq\varliminf\limits_{k\rightarrow\infty}\sum\limits_{n=1}^{\infty}b_{n}\frac{(n+k)!\Gamma(k+\beta_j+2)}{k!\Gamma(n+k+\beta_j+2)}
=-b_{0}<\infty,$$
where the first inequality follows from the Fatou's lemma and the equality follows from (\ref{equ311}). Therefore, $\{b_{n}\}_{n\geq0}$ is summable. By Lemma \ref{thm204}, $b_{n}=0\ (n\geq0)$, which means that if $\alpha_{j}\neq0$, then $c_{\alpha}=0$. Thus $j\notin\mathbb{N}_{UF}$, which implies $B\setminus\mathbb{N}_F\subset B\setminus\mathbb{N}_{UF}$, and hence $\mathbb{N}_{UF}\cap B\subset\mathbb{N}_F\cap B$.

Noticing that $U: M\rightarrow N$ is a unitary module map, the opposite inclusion can be obtained by the same argument, and the proof is complete.
\end{proof}

We are ready to prove Lemma \ref{thm309}.

\vskip2mm

\noindent\textbf{Proof of Lemma \ref{thm309}.} It follows from Lemma \ref{thm310} that $\mathbb{N}_{UF}\cap B=\mathbb{N}_F\cap B$ is a finite set. We denote it by $X$.

When $X$ is empty, it follows that $F,UF\in H^{2}(\mathbb{D}_{2}^{H})$, which implies $C_\alpha F=F$ and $C_\alpha UF=UF$. Therefore
$$\|C_\alpha F\|=\|F\|=\|UF\|=\|C_\alpha UF\|,$$
which ensures the desired conclusion.

Now we assume $X$ is nonempty. Set
$$\mathbb{Z}_{+}^{X}=\{\alpha\in\mathbb{Z}_{+}^{\infty}: \alpha_{n}=0\;\text{for all}\;n\notin X\}.$$
By the definition of $X$, if $\alpha\in\mathbb{Z}_{+}^{B}\setminus \mathbb{Z}_{+}^{X}$, then $C_{\alpha}F=C_{\alpha}UF=0$. Therefore, it suffices to treat the case $\alpha\in\mathbb{Z}_{+}^{X}$. Suppose $X=\{j_1,j_2,\cdots,j_s\}$. It is easily seen that for each index $\gamma\in\mathbb{Z}_{+}^{X}$,
\begin{equation}\label{equ312}
\begin{aligned}\|\zeta^{\gamma}F\|^{2}&=\sum_{\alpha\in\mathbb{Z}_{+}^{X}}\|C_\alpha F\|^{2}\|\zeta^{\alpha+\gamma}\|^{2}
\\&=\sum_{\alpha_{j_{1}}=0}^{\infty}\cdots\sum_{\alpha_{j_{s}}=0}^{\infty}\left[\|C_\alpha F\|^{2}
\prod_{l=1}^{s}\frac{(\alpha_{j_{l}}+\gamma_{j_{l}})!\Gamma(\beta_{j_{l}}+2)}{\Gamma(\alpha_{j_{l}}+\gamma_{j_{l}}+\beta_{j_{l}}+2)}\right].\end{aligned}
\end{equation}
Similarly,
\begin{equation}\label{equ313}
\|\zeta^{\gamma}UF\|^{2}=\sum_{\alpha_{j_{1}}=0}^{\infty}\cdots\sum_{\alpha_{j_{s}}=0}^{\infty}\left[\|C_\alpha UF\|^{2}
\prod_{l=1}^{s}\frac{(\alpha_{j_{l}}+\gamma_{j_{l}})!\Gamma(\beta_{j_{l}}+2)}{\Gamma(\alpha_{j_{l}}+\gamma_{j_{l}}+\beta_{j_{l}}+2)}\right].
\end{equation}
Since
$$\|\zeta^{\gamma}F\|=\|U(\zeta^{\gamma}F)\|=\|\zeta^{\gamma}UF\|,$$
it follows from (\ref{equ312}) and (\ref{equ313}) that for every $\gamma_{j_{1}},\cdots,\gamma_{j_{s}}\geq0$,
$$\sum_{\alpha_{j_{1}}=0}^{\infty}\cdots\sum_{\alpha_{j_{s}}=0}^{\infty}\left[\left(\|C_\alpha F\|^{2}-\|C_\alpha UF\|^{2}\right)
\prod_{l=1}^{s}\frac{(\alpha_{j_{l}}+\gamma_{j_{l}})!\Gamma(\beta_{j_{l}}+2)}{\Gamma(\alpha_{j_{l}}+\gamma_{j_{l}}+\beta_{j_{l}}+2)}\right]=0.$$
Reusing Lemma \ref{thm204} $s$ times yields for every $\alpha\in\mathbb{Z}_{+}^{X}$, $\|C_\alpha F\|=\|C_\alpha UF\|$, which gives the desired conclusion. $\hfill\square$

\vskip2mm

The following argument will be used in the proof of Theorem \ref{thm301}.

Given a sequence $\bm{\beta}=\{\beta_{n}\}_{n\in\mathbb{N}}$ in $[-1,\infty)$, the corresponding $H=\{n\in\mathbb{N}: \beta_{n}=-1\}$, $B=\{n\in\mathbb{N}: \beta_{n}\neq-1\}$. For each subset $\mathcal{F}$ of $A_{\bm{\beta}}^2$, set
$$\mathcal{E}_{\mathcal{F}}=\{C_{\alpha}F: \alpha\in\mathbb{Z}_{+}^{B}, F\in\mathcal{F}\}.$$
We claim that when $\mathcal{I}$ is an ideal of $\mathcal{P}_{\infty}$, $\mathcal{E}_{\mathcal{I}}$ is an ideal of $\mathcal{P}_{H}$. Indeed, it is clear that for $\varphi\in\mathcal{P}_{H}$, $p\in\mathcal{I}$ and $\alpha\in\mathbb{Z}_{+}^{B}$,
$$\varphi C_{\alpha}p=C_{\alpha}(\varphi p)\in\mathcal{E}_{\mathcal{I}}.$$
It is remaining to show that $\mathcal{E}_{\mathcal{I}}$ is closed under addition. For $p,q\in\mathcal{I}$ and $\alpha,\beta\in\mathbb{Z}_{+}^{B}$, a direct verification shows that
$$C_{\alpha}p+C_{\beta}q=C_{\alpha+\beta}(\zeta^{\beta}p+\zeta^{\alpha}q)\in\mathcal{E}_{\mathcal{I}}.$$
This gives the desired conclusion.

Now we proceed to present the proof of Theorem \ref{thm301}.

\vskip2mm

\noindent\textbf{Proof of Theorem \ref{thm301}.} $(1)\Rightarrow(2)$: Suppose $M$ is unitarily equivalent to $N$ via a unitary module map $U$. Put $\mathcal{I}=M\cap\mathcal{P}_{\infty}$, $\mathcal{J}=N\cap\mathcal{P}_{\infty}$. The above argument shows that $\mathcal{E}_{\mathcal{I}}$ and $\mathcal{E}_{\mathcal{J}}$ are ideals of $\mathcal{P}_{H}$. As we mentioned previously, $[\mathcal{E}_{\mathcal{I}}]_{H}$ and $[\mathcal{E}_{U\mathcal{I}}]_{H}$ are submodules of $H^2(\mathbb{D}_{2}^{H})$. Define $U_{1}: \mathcal{E}_{\mathcal{I}}\rightarrow\mathcal{E}_{U\mathcal{I}}$ by putting
$$U_{1}(C_{\alpha}w)=C_{\alpha}Uw,\quad \alpha\in\mathbb{Z}_{+}^{B},\ w\in\mathcal{I}.$$
We first show that $U_{1}$ is linear. Indeed, for $\alpha_1,\alpha_2\in\mathbb{Z}_{+}^{B}$ and $w_1,w_2\in\mathcal{I}$,
$$\begin{aligned}U_1(C_{\alpha_1}w_1+C_{\alpha_2}w_2)&=U_1C_{\alpha_1+\alpha_2}(\zeta^{\alpha_2}w_1+\zeta^{\alpha_1}w_2)
\\&=C_{\alpha_1+\alpha_2}(\zeta^{\alpha_2}Uw_1+\zeta^{\alpha_1}Uw_2)
\\&=C_{\alpha_1}Uw_1+C_{\alpha_2}Uw_2
\\&=U_1C_{\alpha_1}w_1+U_1C_{\alpha_2}w_2.\end{aligned}$$
The claim follows. Moreover, it follows from Lemma \ref{thm309} that $U_{1}$ is an isometry. Since $\mathcal{E}_{\mathcal{I}}$ and $\mathcal{E}_{U\mathcal{I}}$ are dense in $[\mathcal{E}_{\mathcal{I}}]_{H}$ and $[\mathcal{E}_{U\mathcal{I}}]_{H}$, respectively, $U_{1}$ can be extended to a unitary module map from $[\mathcal{E}_{\mathcal{I}}]_{H}$ onto $[\mathcal{E}_{U\mathcal{I}}]_{H}$. We also denote it by $U_{1}$.

For each $F\in[\mathcal{E}_{\mathcal{I}}]$, we define
$$U_2F=\sum_{\alpha\in\mathbb{Z}_{+}^{B}}U_1C_\alpha F\cdot\zeta^\alpha.$$
A similar argument as in the proof of Proposition \ref{thm306} yields $U_{2}: [\mathcal{E}_{\mathcal{I}}]\rightarrow[\mathcal{E}_{U\mathcal{I}}]$ is well defined and unitary. Moreover, if $h\in\mathcal{I}$, then
$$U_2h=\sum_{\alpha\in\mathbb{Z}_{+}^{B}}U_1C_\alpha h\cdot\zeta^\alpha=\sum_{\alpha\in\mathbb{Z}_{+}^{B}}C_\alpha Uh\cdot\zeta^\alpha=Uh.$$
Noticing that $\mathcal{I}$ is dense in $M$, the above equality implies that $U_2|_M=U$.

Since both $\mathcal{E}_{\mathcal{J}}$ and $\mathcal{E}_{U\mathcal{I}}$ are dense in $\mathcal{E}_{N}$, we conclude that $[\mathcal{E}_{\mathcal{J}}]=[\mathcal{E}_{U\mathcal{I}}]$. Suppose the greatest common divisors of ideals $[\mathcal{E}_{\mathcal{I}}]\cap\mathcal{P}_{\infty}$ and $[\mathcal{E}_{\mathcal{J}}]\cap\mathcal{P}_{\infty}$ are $r,s\in\mathcal{P}_{H}$,  respectively. By Proposition \ref{thm306}, $U_{2}$ has a unitary extension $U_{3}: [r]\rightarrow[s]$. Clearly, when $H$ is empty, $\mathcal{P}_{H}$ is the complex field $\mathbb{C}$, and hence $r\asymp s$. When $H$ is nonempty, take $S=H$ in Lemma \ref{thm307}, we get for $n$ sufficiently large, $r,s\in\mathcal{P}_{H(n)}$, and $[r]_{H(n)}$ is unitarily equivalent to $[s]_{H(n)}$ as submodules of $H^{2}(\mathbb{D}^{H(n)})$. Then it follows from \cite[Corollary 3.6]{Guo1} that $r\asymp s$.

Note that $p,r$ are the greatest common divisors of $M\cap\mathcal{P}_{\infty}$ and $[\mathcal{E}_{\mathcal{I}}]\cap\mathcal{P}_{\infty}$, respectively, and $M\cap\mathcal{P}_{\infty}=\mathcal{I}\subset[\mathcal{E}_{\mathcal{I}}]\cap\mathcal{P}_{\infty}$. It follows immediately $r|p$, and thus
$$\mathcal{I}=M\cap\mathcal{P}_{\infty}=r\varphi\mathcal{K},$$
where $\varphi\in\mathcal{P}_{\infty}$, and $\mathcal{K}$ is an ideal of $\mathcal{P}_{\infty}$ without nontrivial common divisor. Similarly, there is a polynomial $\psi\in\mathcal{P}_{\infty}$ and an ideal $\mathcal{L}$ of $\mathcal{P}_{\infty}$ without nontrivial common divisor such that
$$\mathcal{J}=N\cap\mathcal{P}_{\infty}=s\psi\mathcal{L}.$$
Therefore $M=[r\varphi\mathcal{K}]$ and $N=[s\psi\mathcal{L}]$, which in turn imply
$$N=U_{3}M=U_{3}[r\varphi\mathcal{K}]=[s\varphi\mathcal{K}],$$
and thus
$$s\varphi\mathcal{K}\subset N\cap\mathcal{P}_{\infty}=\mathcal{J}=s\psi\mathcal{L},$$
forcing $\varphi\mathcal{K}\subset\psi\mathcal{L}$. By a similar argument, we obtain the opposite inclusion, and thus $\varphi\mathcal{K}=\psi\mathcal{L}$. Noticing that $\mathcal{K},\mathcal{L}$ are ideals without nontrivial common divisors, we have $\varphi=\psi$ and $\mathcal{K}=\mathcal{L}$.

$(2)\Rightarrow(3)$: Since $r\asymp s$, there are two polynomials $u,v\in\mathcal{P}_{H}$ without zero in $\mathbb{D}^{\infty}$, such that $|ru|=|sv|$ on $\mathbb{T}^{\infty}$. For each $f\in\mathcal{K}$, there is an integer $k$ such that $p,u,f\in\mathcal{P}_{k}$. By Theorem \ref{thm201},
$$pf=puf\cdot\frac{1}{u}\in[puf]_k\subset[puf]\subset[pu\mathcal{K}],$$
and hence $[p\mathcal{K}]\subset[pu\mathcal{K}]$. Since $u$ is a polynomial and $p=r\varphi$,
$$M=[p\mathcal{K}]=[pu\mathcal{K}]=[ru\varphi\mathcal{K}].$$
Similarly,
$$N=[q\mathcal{K}]=[qv\mathcal{K}]=[sv\varphi\mathcal{K}].$$
Put $\tilde{p}=ru$, $\tilde{q}=sv$, and $\mathcal{G}=\varphi\mathcal{K}$. It follows that $|\tilde{p}|=|\tilde{q}|$ on $\mathbb{T}^\infty$, and $M,N$ are the closures of $\tilde{p}\mathcal{G}$ and $\tilde{q}\mathcal{G}$, respectively.

$(3)\Rightarrow(1)$: Define a bounded linear operator $U$ of $M$ onto $N$ by putting
$$U(\tilde{p}g)=\tilde{q}g,\quad g\in\mathcal{G}.$$
Since $|\tilde{p}|=|\tilde{q}|$ on $\mathbb{T}^{\infty}$, for each $g\in\mathcal{G}$,
$$\|\tilde{p}g\|^{2}=\sum_{\alpha\in\mathbb{Z}_{+}^{B}}\|\tilde{p}C_{\alpha}g\|^{2}\|\zeta^{\alpha}\|^{2}
=\sum_{\alpha\in\mathbb{Z}_{+}^{B}}\|\tilde{q}C_{\alpha}g\|^{2}\|\zeta^{\alpha}\|^{2}=\|\tilde{q}g\|^{2}.$$
Therefore, $U: M\rightarrow N$ is a unitary module map, and thus $M$ and $N$ are unitarily equivalent. $\hfill\square$

\section{Unitary equivalence and unitary orbits of polynomially generated principle submodules}

In this section, we turn to the case of principle submodules of $A_{\bm{\beta}}^2$ generated by a single polynomial.

Given a polynomial $p\in\mathcal{P}_{\infty}$,  we first consider the structure of ideal $[p]\cap\mathcal{P}_{\infty}$. To do this, we need the following lemma, which is a direct corollary of \cite[Theorem 1.3.2]{Ru}. For every polynomial $\varphi$ in $n$ complex variables, we use $Z(\varphi)$ to denote the zero set of $\varphi$, that is
$$Z(\varphi)=\left\{\zeta\in\mathbb{C}^{n}: \varphi(\zeta)=0\right\}.$$

\begin{lem}\label{thm400}
Suppose $\varphi$ and $\psi$ are relatively prime polynomials in $n$ complex variables. If $\frac{\varphi}{\psi}$ is holomorphic in $\mathbb{D}^n$, then $\psi$ has no zero in $\mathbb{D}^n$.
\end{lem}

\begin{proof}
If $n=1$, the conclusion is trivial. Below we suppose $n>1$. Since $\frac{\varphi}{\psi}$ is holomorphic in $\mathbb{D}^n$, this ensures $Z(\psi)\cap\mathbb{D}^n\subset Z(\varphi)\cap\mathbb{D}^n$. To obtain a contradiction, we assume that there is a point $\lambda\in\mathbb{D}^n$ such that $\psi(\lambda)=0$, and thus $\varphi(\lambda)=0$.

Put $\tilde{\varphi}=\varphi(\cdot+\lambda)$, $\tilde{\psi}=\psi(\cdot+\lambda)$, and
$$-\lambda+\mathbb{D}^n=\left\{-\lambda+\zeta: \zeta\in\mathbb{D}^n\right\}.$$
Then $\tilde{\varphi}$ and $\tilde{\psi}$ are relatively prime, and $\tilde{\varphi}(0)=\tilde{\psi}(0)=0$. It follows from \cite[Theorem 1.3.2]{Ru} that $Z(\tilde{\psi})\cap(-\lambda+\mathbb{D}^n)$ is not a subset of  $Z(\tilde{\varphi})\cap(-\lambda+\mathbb{D}^n)$. This implies that $Z(\psi)\cap\mathbb{D}^n$ is not a subset of  $Z(\varphi)\cap\mathbb{D}^n$, which yields a contradiction, and hence Lemma \ref{thm400} follows.
\end{proof}

Observing that when $p\in\mathcal{P}_\infty$ has no zero in $\mathbb{D}^{\infty}$, Theorem \ref{thm201} shows that the submodule $[p]$ generated by $p$ is the whole space. In what follows, we consider the case of $p$ having  zero points in $\mathbb{D}^{\infty}$. Given such a polynomial $p$,  we decompose $p=p_1p_2$ to be such that $p_1$ has no zero in $\mathbb{D}^{\infty}$, and each nontrivial irreducible factor of $p_2$ has zeros in $\mathbb{D}^{\infty}$. Write $p_{*}=p_2$. Clearly, $p_{*}$ is unique up to a constant factor.

Now we can give a clear expression of submodule $[p]\cap\mathcal{P}_{\infty}$ as follows:
\begin{equation}\label{equ401}
[p]\cap\mathcal{P}_\infty=p_{*}\mathcal{P}_\infty.
\end{equation}
To see this, for each $r\in[p]\cap\mathcal{P}_{\infty}$, it follows from Theorem \ref{thm201} that $\frac{r}{p}$ is holomorphic in $\mathbb{D}^{\mathbb{N}_r\cup\mathbb{N}_p}$. Combining this with Lemma \ref{thm400} implies that $r$ can be divided by $p_{*}$ in $\mathcal{P}_{\infty}$, which yields $[p]\cap\mathcal{P}_{\infty}\subset p_{*}\mathcal{P}_\infty$. On the other hand, if $\varphi\in\mathcal{P}_{\infty}$ and $r=p_{*}\varphi$, then $\frac{r}{p}$ is holomorphic in $\mathbb{D}^{\mathbb{N}_r\cup\mathbb{N}_p}$. By Theorem \ref{thm201}, $r\in[p]$. Therefore $p_{*}\mathcal{P}_\infty\subset[p]\cap\mathcal{P}_{\infty}$, and the desired result follows.

Combining (\ref{equ401}) with Theorem \ref{thm301}, we have the following result on polynomially generated principle submodules.

\begin{thm}\label{thm401}
Let $[p]$, $[q]$ be two principle submodules of $A_{\bm{\beta}}^2$ generated by polynomials $p$, $q$, respectively. Then
they are unitarily equivalent if and only if the polynomials $p_{*},q_{*}$ admit factorizations $p_{*}=r\varphi$, $q_{*}=s\varphi$, where $\varphi\in\mathcal{P}_{\infty}$, $r,s\in\mathcal{P}_{H}$ and $r\asymp s$.
\end{thm}

Now we turn to unitary orbits of polynomially generated principle submodules. For each submodule $M$ of $A_{\bm{\beta}}^2$, recall that its unitary orbit, denoted by $\mathrm{orb}_{u}(M)$, consists of all submodules of $A_{\bm{\beta}}^2$ which are unitarily equivalent to $M$. For $p\in\mathcal{P}_{\infty}$, we consider submodules in $\mathrm{orb}_{u}([p])$.

Recall that for a sequence $\bm{\beta}=\{\beta_{n}\}_{n\in\mathbb{N}}$ in $[-1,\infty)$, the corresponding $H=\{n\in\mathbb{N}: \beta_{n}=-1\}$, and $B=\{n\in\mathbb{N}: \beta_{n}\neq-1\}$. Suppose $M\in\mathrm{orb}_{u}([p])$ and $U: [p]\rightarrow M$ is a unitary module map. It follows that $M=U[p]=[Up]$, and hence $M$ can be determined uniquely by $Up$. An analysis similar to that in the proof of \cite[Theorem 3.4.1]{CG} shows that for every $\alpha\in\mathbb{Z}_{+}^{B}$,
\begin{equation}\label{equ402}
|(C_{\alpha}Up)^{*}|=|(C_{\alpha}p)^{*}|
\end{equation}
almost everywhere on $\mathbb{T}^{H}$ with respect to its Haar measure $\sigma_{H}$, where $(C_{\alpha}Up)^{*}$ and $(C_{\alpha}p)^{*}$ are the boundary-value functions of $C_{\alpha}Up$ and $C_{\alpha}p$, respectively. Indeed, for each $\varphi\in\mathcal{P}_{H}$,
$$\|\varphi  C_{\alpha}Up\|=\|C_{\alpha}(\varphi Up)\|=\|C_{\alpha}U(\varphi p)\|=\|C_{\alpha}(\varphi p)\|=\|\varphi C_{\alpha}p\|,$$
where the third equality follows from Lemma \ref{thm309}. Combining this with (\ref{equ315}) implies
$$\int_{\mathbb{T}^{H}}|\varphi|^2(|(C_{\alpha}Up)^{*}|^{2}-|(C_{\alpha}p)^{*}|^{2})d\sigma_{H}=0.$$
Thus, for any $\varphi,\psi\in\mathcal{P}_{H}$,
$$\int_{\mathbb{T}^{H}}\varphi\overline{\psi}(|(C_{\alpha}Up)^{*}|^{2}-|(C_{\alpha}p)^{*}|^{2})d\sigma_{H}=0.$$
By virtue of the Stone-Weierstrass theorem, the regular Borel measure $(|(C_{\alpha}Up)^{*}|^{2}-|(C_{\alpha}p)^{*}|^{2})d\sigma_{H}$ annihilates $C(\mathbb{T}^H)$, the algebra of all continuous functions on $\mathbb{T}^H$. Therefore, the desired conclusion can be obtained by applying the Riesz representation theorem.

In particular, when $p$ is a monomial, Theorem \ref{thm402} below gives a complete characterization of $\mathrm{orb}_{u}([p])$ by determining the form of $Up$. Recall that $F\in H^{2}(\mathbb{D}_{2}^{H})$ is an inner function, if $F$ is bounded on $\mathbb{D}_{2}^{H}$ and $|F^{*}|=1$ almost everywhere on $\mathbb{T}^{H}$ with respect to $\sigma_{H}$, where $F^*$ is the boundary-value function of $F$. It is clear that for each inner function $F\in H^{2}(\mathbb{D}_{2}^{H})$ and $G\in A_{\bm{\beta}}^{2}$, we have
$$\|FG\|^2=\sum\limits_{\alpha\in\mathbb{Z}_{+}^{B}}\|FC_{\alpha}G\|^2\|\zeta^{\alpha}\|^2
=\sum\limits_{\alpha\in\mathbb{Z}_{+}^{B}}\|C_{\alpha}G\|^2\|\zeta^{\alpha}\|^2=\|G\|^2.$$

\begin{thm}\label{thm402}
Suppose $\gamma\in\mathbb{Z}_{+}^{\infty}$. Then
$$\mathrm{orb}_{u}([\zeta^\gamma])=\left\{F[\zeta^{\tilde{\gamma}}]: F\in H^{2}(\mathbb{D}_{2}^{H})\;\text{\emph{is inner}}\right\},$$
where
$$\tilde{\gamma}_{n}=\begin{cases} \gamma_{n} & n\in B, \\
0 & n\notin B. \end{cases}$$
\end{thm}

\begin{proof}
Assume $F\in H^{2}(\mathbb{D}_{2}^{H})$ is inner. Notice that $[\zeta^\gamma]=\zeta^\gamma A_{\bm{\beta}}^2$ and $F[\zeta^{\tilde{\gamma}}]=\zeta^{\tilde{\gamma}}F A_{\bm{\beta}}^2$. We define a linear operator $U:[\zeta^\gamma]\rightarrow F[\zeta^{\tilde{\gamma}}]$ by putting
$$U(\zeta^\gamma G)=\zeta^{\tilde{\gamma}}FG,\quad G\in A_{\bm{\beta}}^2.$$
It follows easily that $U$ is a unitary module map, and hence $F[\zeta^{\tilde{\gamma}}]\in\mathrm{orb}_{u}([\zeta^\gamma])$.

Conversely, we suppose that $M\in\mathrm{orb}_{u}([\zeta^\gamma])$ and $U: [\zeta^\gamma]\rightarrow M$ is a unitary module map. Since $\zeta^{\gamma}=\zeta^{\gamma-\tilde{\gamma}}\cdot\zeta^{\tilde{\gamma}}$, it follows that
$$C_{\alpha}\zeta^{\gamma}=\begin{cases} \zeta^{\gamma-\tilde{\gamma}} & \alpha=\tilde{\gamma}, \\
0 & \alpha\neq\tilde{\gamma}.\end{cases}$$
Put $F=C_{\tilde{\gamma}}U\zeta^\gamma\in H^{2}(\mathbb{D}_{2}^{H})$. Then (\ref{equ402}) implies that $|F^{*}|=1$ almost everywhere on $\mathbb{T}^{H}$ with respect to $\sigma_{H}$, and $C_{\alpha}U\zeta^{\gamma}=0 \;(\alpha\neq\tilde{\gamma})$. These yield $F\in H^{2}(\mathbb{D}_{2}^{H})$ is inner, and
$$U\zeta^\gamma=\sum\limits_{\alpha\in\mathbb{Z}_{+}^{B}}C_{\alpha}U\zeta^\gamma\cdot\zeta^{\alpha}=F\zeta^{\tilde{\gamma}}.$$
Therefore,
$$M=U[\zeta^\gamma]=[U\zeta^\gamma]=[F\zeta^{\tilde{\gamma}}]=F[\zeta^{\tilde{\gamma}}].$$
The proof is complete.
\end{proof}

For a general polynomial $p\in\mathcal{P}_{\infty}$,   it seems to be difficult how to characterize the unitary orbit of $[p]$ in $A_{\bm{\beta}}^{2}$.

\vskip3mm \noindent{Hui Dan, College of Mathematics, Sichuan University, Chengdu, Sichuan,
610065, China, E-mail:  danhuimath@gmail.com% Dan Hui,

\noindent Kunyu Guo, School of Mathematical Sciences, Fudan
University, Shanghai, 200433, China, E-mail: kyguo@fudan.edu.cn

\noindent Jiaqi Ni, School of Mathematical Sciences, Fudan
University, Shanghai, 200433, China, E-mail: jqni15@fudan.edu.cn

\end{document}